\documentclass[a4j, 12pt, draft]{amsart}
\usepackage{url}
\usepackage{amssymb}
\usepackage{enumerate}
\usepackage{braket}
\usepackage{ascmac}
\usepackage{amsmath}
\usepackage{mathrsfs}	
\usepackage{bm}
\usepackage{color}

\allowdisplaybreaks[4]

\theoremstyle{definition}
\newtheorem{definition}{Definition}[section]

\newtheorem*{ackname}{Acknowledgment}

\theoremstyle{theorem}
\newtheorem{theorem}{Theorem}

\newtheorem{lemma}
{Lemma}
\newtheorem{proposition}[definition]{Proposition}

\newtheorem{problem}{Problem}
\newtheorem*{theorem*}{Theorem}

\let\sl\l
\renewcommand\l{%
	\leavevmode
  \ifmmode
    \left
  \else
    \sl
  \fi}

\renewcommand{\Set}[2]{\left\{ #1 \; \middle| \; #2 \right\}}

\newcommand{\CC}{\mathbb{C}}
\newcommand{\RR}{\mathbb{R}}

\newcommand{\ZZ}{\mathbb{Z}}
\newcommand{\NN}{\mathbb{N}}

\newcommand{\e}{\varepsilon}
\newcommand{\s}{\sigma}

\newcommand{\Lam}{\Lambda}
\newcommand{\I}{\mathcal{I}}

\newcommand{\qquadt}{\qquad \qquad \qquad}
\newcommand{\qquadf}{\qquad \qquad \qquad \qquad}
\newcommand{\us}{\underset}
\newcommand{\te}{\operatorname{\tilde{\eta}}}
\newcommand{\meas}{\operatorname{meas}}
\newcommand{\Li}{\operatorname{Li}}

\renewcommand{\a}{\alpha}
\renewcommand{\b}{\beta}
\renewcommand{\r}{\right}

\renewcommand{\epsilon}{\varepsilon}
\renewcommand{\bar}{\overline}
\renewcommand{\phi}{\varphi}

\begin{document}

\title[Denseness of iterated integrals of $\log{\zeta(s)}$]{On the value-distribution of iterated integrals of the logarithm 
of the Riemann zeta-function I: denseness
}
\author[K.~Endo and S. Inoue]{Kenta Endo and Sh\={o}ta Inoue}

\address[K.~Endo]{Graduate School of Mathematics, Nagoya University, Chikusa-ku, Nagoya 464-8602, Japan.}
\email{m16010c@math.nagoya-u.ac.jp}

\address[S. Inoue]{Graduate School of Mathematics, Nagoya University, Chikusa-ku, Nagoya 464-8602, Japan.}
\email{m16006w@math.nagoya-u.ac.jp}



\maketitle

\begin{abstract}
We consider iterated integrals of $\log\zeta(s)$ on certain vertical and horizontal lines.
Here, the function $\zeta(s)$ is the Riemann zeta-function.
It is a well known open problem whether or not the values of the Riemann zeta-function on the critical line 
are dense in the complex plane.
In this paper, 
we give a result for the denseness of the values of the iterated integrals on the horizontal lines.
By using this result, 
we obtain the denseness of the values of $\int_{0}^{t} \log \zeta(1/2 + it')dt'$ under the Riemann Hypothesis.
Moreover, we show that, for any $m\geq 2$, the denseness of the values of an $m$-times iterated  integral on the critical line is equivalent to the Riemann Hypothesis.
\end{abstract}


\section{\textbf{Introduction and statement of results}}
In the present paper, we give some results for the value-distribution of iterated integrals of the logarithm of 
the Riemann zeta-function $\zeta(s)$.
Many mathematicians have studied the value-distribution of the Riemann zeta-function and other $L$-functions.
Here we should mention two remarkable results which are starting points of those studies.

\begin{theorem*}[Bohr and Courant in 1914 \cite{BC1914}]
For fixed $\frac{1}{2} < \s \leq 1$, the set $\l\{ \zeta(\s + it) \mid t \in \RR \r\}$ is dense in the complex plane.
\end{theorem*}

\begin{theorem*}[Bohr in 1916 \cite{B1916}]
For fixed $\frac{1}{2} < \s \leq 1$, the set $\l\{ \log{\zeta(\s + it)} \mid t \in \RR \r\}$ is dense in the complex plane.
\end{theorem*}

Note that the latter theorem is an improvement of former one since the former 
is an immediate consequence from the latter.
As developments of these theorems, the Bohr-Jessen limit theorem \cite{BJ1930}, Selberg's limit theorem \cite{S1992},
and Voronin's universality theorem \cite{V1975} are well known.
By these theorems, we can understand some properties of $\zeta(s)$ such as the exact value-distribution of $\zeta(s)$ and 
the complexity of the behavior of $\zeta(s)$ in the critical strip.
As further developments of these results, 
there are many studies such as \cite{BH1995}, \cite{HM1999}, \cite{JW1935}, \cite{LLR2019}, \cite{RA2011}.

Here, we mention some known facts for the denseness of the values $\zeta(\s+it)$ 
for $t \in \RR$.
In the case $\s > 1$ fixed, the values $\zeta(\s+it)$ is bounded.
As for the case $\sigma < 1/2$, 
it has been proved by Garunk\u{s}tis and Steuding \cite{GS2014} that the values $\zeta(\sigma + it)$ for $t \in \RR$ are not dense in the complex plane under the Riemann Hypothesis. Additionally, as we mentioned above, the denseness in the case $1/2 < \s \leq 1$ has been proved.
Hence, the remaining problem for the denseness is only the following.

\begin{problem}	\label{DP}
Is the set $\l\{ \log{\zeta(1/2 + it)} \mid t \in \RR \r\}$ dense in the complex plane?
\end{problem}

For Problem \ref{DP}, there is an interesting study by Kowalski and Nikeghbali \cite{KN2012}.
They studied the Fourier transform of the probability measure which represents the probability of $\log{\zeta(1/2+it)} \in A$ 
with $A$ a Borel set. 
In particular, they gave a sufficient condition that the values $\zeta(1/2+it)$ for $t \in \mathbb{R}$ 
are dense in the complex plane (see \cite[Corollary 9]{KN2012}).
Hence, from their study, we might guess that the answer for Problem \ref{DP} could be yes.
However, as they mentioned in their paper \cite{KN2012}, their sufficient condition is rather strong. 
Therefore, it is also not strange that the answer for Problem \ref{DP} could be no.
Moreover, Garunk\u{s}tis and Steuding \cite{GS2014} showed that 
the set of $(\zeta(1/2+it), \zeta'(1/2+it))$ for $t \in \mathbb{R}$ is not dense in $\CC^{2}$.
As we can see from these works, it seems difficult to decide clearly the answer of Problem \ref{DP} at present. 
Hence, it is desirable to obtain some new information for this problem, 
and we give a new information for this problem in this paper.

In order to give new information of this theme,
we consider the function $\eta_{m}(s)$ defined by
\begin{align*}
\eta_{m}(\s + it)
= \int_{0}^{t}\eta_{m-1}(\s + it')dt' + c_{m}(\s), 
\end{align*}
where 
\begin{gather*}
\eta_{0}(\s + it)
= \log{\zeta(\s + it)}, \\
c_{m}(\s)
= \frac{i^{m}}{(m-1)!}\int_{\s}^{\infty}(\a - \s)^{m-1}\log{\zeta(\a)}d\a.
\end{gather*}
The second author studied this function and gave some results in \cite{II2019}.
In the present paper we discuss the topic related to \cite[Section 2.4]{II2019}.
Since the function $\eta_{m}(s)$ is the $m$-times iterated integral of $\log{\zeta(s)}$ on the vertical line, 
we can expect that the function has the information of the value-distribution of $\log{\zeta(s)}$.
In particular, 
since $\eta_{m}(1/2+it)$ is the iterated integral on the critical line, 
the study of the value-distribution of this function could be expected to give the new information on Problem \ref{DP}.
From this background, we study the value-distribution of the function $\eta_{m}(s)$ to give the following theorem. 

\begin{theorem}	\label{DTetam}Let $1/2 \leq \s < 1$. 
If the number of zeros $\rho = \b + i\gamma$ of $\zeta(s)$ with $\b > \s$ is finite, then the set 
\begin{align*}
\l\{ \int_{0}^{t}\log{\zeta(\s + it')}dt' \; \Big{|} \; t \in [0, \infty) \r\}
\end{align*}
is dense in the complex plane.
Moreover, for each integer $m \geq 2$, the following statements are equivalent.
\begin{itemize}
\item[(I)] The Riemann zeta-function does not have any zero whose real part is greater than $\s$.
\item[(II)] The set
$
\l\{ \eta_{m}(\s + it) \mid t \in [0, \infty) \r\}
$
is dense in the complex plane.
\end{itemize}
\end{theorem}

From this theorem, we see that the Riemann Hypothesis implies that the set
\begin{align*}
\l\{ \int_{0}^{t}\log{\zeta(1/2 + it')}dt' \; \Big{|} \; t \in [0, \infty) \r\}
\end{align*}
is dense in the complex plane. 
This implication seems to suggest that the answer of Problem \ref{DP} is yes.
Moreover, the equivalence as above would be a new type of statement which gives the relation between the denseness of values of 
the Riemann zeta-function and the Riemann Hypothesis. 

Here, we mention the plan of the proof of Theorem \ref{DTetam} briefly. 
We introduce the function $\tilde{\eta}_{m}(\s+it)$ recursively by  
\begin{align*}
\tilde{\eta}_{m}(\s+it) = \int_{\s}^{\infty}\tilde{\eta}_{m-1}(\a+it)d\a, 
\end{align*}
where $\tilde{\eta}_{0}(\s+it) = \log{\zeta(\s+it)}$.
This function is the $m$-times iterated integral of $\log{\zeta(\s+it)}$ on the horizontal line. 
The greatest interest of ours in this paper is the answer of Problem \ref{DP} and the value-distribution of $\eta_m(1/2 + it)$.
However, the function $\te_m(s)$ is regular in the same region as in the case of $\log{\zeta(s)}$, 
and also some properties of this function are similar to $\log \zeta(s)$.
From this observation, 
this function would be an interesting object itself,
and we obtain the following theorem unconditionally.

\begin{theorem}	\label{DTtetam}
Let $1/2 \leq \s < 1$, and $m$ be a positive integer. 
Let $T_{0}$ be any positive number.
Then the set
\begin{align*}
\l\{ \te_{m}(\s+it) \mid t \in [T_{0}, \infty) \r\}
\end{align*}
is dense in the complex plane.
\end{theorem}

Theorem \ref{DTetam} can be obtained from Theorem \ref{DTtetam} and the following lemma.

\begin{lemma}	\label{BFQm}
Let $m$ be a positive integer, and let $t > 0$.
Then, for any $\s \geq 1/2$, we have
\begin{align*}
\eta_{m}(\s + it)
= i^{m}\te_{m}(\s+it)
+ 2\pi \sum_{k = 0}^{m-1}\frac{i^{m-1-k}}{(m-k)! k!}\us{\b > \s}{\sum_{0 < \gamma < t}}(\b - \s)^{m-k}(t - \gamma)^{k}.
\end{align*}
\end{lemma}

This lemma immediately follows from Lemma $1$ in \cite{II2019} and the equation

\begin{align}	\label{etatF}
\te_{m}(\s+it)
= \frac{1}{(m-1)!}\int_{\s}^{\infty}(\a - \s)^{m-1}\log{\zeta(\a + it)}d\a.
\end{align}
This equation can be obtained easily by using integration by parts.
Hence, our first purpose is to show Theorem \ref{DTtetam}. 
In the proof of Theorem \ref{DTtetam}, the following two propositions play an important role.  

In the following, the symbol $\meas(\cdot)$ stands for the one-dimensional Lebesgue measure,
and $\Li_{m}(z)$ means the polylogarithmic function defined as $\sum_{n=1}^{\infty}\frac{z^{n}}{n^{m}}$
for $|z| < 1$.

\begin{proposition}	\label{ALD}
Let $m$ be a positive integer.
Then for any $\s \geq 1/2$, $T \geq X^{135}$, $\e > 0$, we have
\begin{align*}
\lim_{X \rightarrow +\infty}\frac{1}{T}
\meas \l\{ t \in [0, T] \; \Bigg{|} \; 
\bigg| \te_{m}(\s + it) - \sum_{p \leq X}\frac{\Li_{m+1}(p^{-\s - it})}{(\log{p})^{m}} \bigg| < \e \r\} = 1.
\end{align*}
\end{proposition}

The important point of this proposition is that $\te_{m}(s)$ can be approximated by the Dirichlet polynomial 
even on the critical line.
To prove this proposition, we must control exactly the contribution of nontrivial zeros of $\zeta(s)$, 
and we therefore need a strong zero density estimate of the Riemann zeta-function like Selberg's result \cite[Theorem 1]{SC1946}.
More precisely, we require that there exist numbers $c > 0$, $A < m + 1$ such that
\begin{align*}
N(\s, T) \ll T^{1 - c(\s - 1/2)} (\log{T})^{A}
\end{align*}
uniformly for $\frac{1}{2} \leq \s \leq 1$. 
Here, $N(\s, T)$ is the number of zeros of $\zeta(s)$ with multiplicity satisfying $\b > \s$ and $0 < \gamma \leq T$.
Therefore, to prove Proposition \ref{ALD}, we need a strong zero density estimate comparable to the assumption 
by Bombieri and Hejhal \cite{BH1995}.
On the other hand, when we discuss the denseness of $\te_{m}(s)$ for fixed $\frac{1}{2} < \s < 1$, 
it suffices to use the weaker estimate
\begin{align*}
N(\s, T) \ll T^{1 - c(\s - 1/2) + \e}
\end{align*}
for every $\e > 0$. 
Hence, there is an essential difference of the depth between the discussion in the case $\frac{1}{2} < \s < 1$ 
and that in the case $\s = \frac{1}{2}$ in Proposition \ref{ALD}.

In contrast, we can prove the following proposition by almost the same method as in \cite{B1916}, \cite{BC1914}.

\begin{proposition}	\label{DLm}
Let m be a positive integer, $1/2\leq \sigma<1$. 
Let $a$ be any complex number, and $\e$ be any positive number.
If we take a sufficiently large number $N_{0} = N_{0}(m, \sigma, a, \varepsilon)$, then, 
for any integer $N \geq N_{0}$, there exists 
some Jordan measurable set $\Theta_0 = \Theta_0(m, \sigma, a, \varepsilon, N) \subset [0, 1)^{\pi(N)}$ with $\meas( \Theta_0 ) >0$ such that 
\begin{align*}
\left| \sum_{p \leq N} \frac{\Li_{m+1}(p^{ - \sigma}\exp(- 2 \pi i \theta_p))}{(\log{p})^{m}} - a \right| < \varepsilon.
\end{align*}
for any $\underline{\theta}= \left(\theta_{p_n}\right)_{n =1}^{\pi(N)} \in \Theta_0$.
\end{proposition}

Roughly speaking, 
Proposition \ref{ALD} means that $\te_{m}(\s+it)$ ``almost" equals the finite sum of polylogarithmic functions when the number of the terms of the sum is sufficiently large,
and Proposition \ref{DLm} that any complex number can be approximated by the finite sum of polylogarithmic functions when the number of the terms of the sum is sufficiently large.

\section{\textbf{Proof of Proposition \ref{ALD}}}

In this section, we prove Proposition \ref{ALD}.
In order to prove it, we prepare two lemmas.

\begin{lemma}	\label{ML}
Let $m$ be a positive integer, and $\s \geq 1/2$.
Let $T$ be large.
Then, for $3 \leq X \leq T^{\frac{1}{135}}$, we have
\begin{align*}
\frac{1}{T}\int_{14}^{T}\bigg| \te_{m}(\s + it)
- \sum_{2 \leq n \leq X}\frac{\Lam(n)}{n^{\s+it}(\log{n})^{m+1}} \bigg|^{2} dt
\ll \frac{X^{1 - 2\s}}{(\log{X})^{2m}}.
\end{align*}
\end{lemma}

\begin{proof}
By Theorem 5 in \cite{II2019}, we have
\begin{align*}
\frac{1}{T}\int_{14}^{T}\bigg| \eta_{m}(\s + it)
- i^{m}\sum_{2 \leq n \leq X}\frac{\Lam(n)}{n^{\s+it}(\log{n})^{m+1}} - Y_{m}(\s+it) \bigg|^{2} dt
\ll \frac{X^{1 - 2\s}}{(\log{X})^{2m}},
\end{align*}
where
\begin{align}	\label{def_Y}
Y_{m}(\s+it) = 2\pi \sum_{k = 0}^{m-1}\frac{i^{m-1-k}}{(m-k)! k!}\us{\b > \s}{\sum_{0 < \gamma < t}}(\b - \s)^{m-k}(t - \gamma)^{k}.
\end{align}
Further, by Lemma \ref{BFQm}, we see that
\begin{align*}
\eta_{m}(\s+it) - Y_{m}(\s+it) = i^{m}\te_{m}(\s+it).
\end{align*}
Hence we obtain this lemma.
\end{proof}

\begin{lemma}	\label{PLvM}
Let $m$ be an integer, $\s \geq 1/2$.
Let $T$ be large. 
Then for $3 \leq X \leq T^{1/4}$, we have
\begin{align*}
\frac{1}{T}\int_{0}^{T}\bigg| \sum_{p \leq X}\frac{\Li_{m+1}(p^{-\s - it})}{(\log{p})^{m}}
- \sum_{2 \leq n \leq X}\frac{\Lam(n)}{n^{\s+it}(\log{n})^{m+1}} \bigg|^2 dt
\ll \frac{X^{1 - 2\s}}{(\log{X})^{2m+1}},
\end{align*}
where the function $\Lam(n)$ is the von Mangoldt function.
\end{lemma}

\begin{proof}
By definitions of the polylogarithmic function and the von Mangoldt function, we find that
\begin{align*}
&\sum_{p \leq X}\frac{\Li_{m+1}(p^{-\s - it})}{(\log{p})^{m}}
- \sum_{2 \leq n \leq X}\frac{\Lam(n)}{n^{\s+it}(\log{n})^{m+1}}
= \sum_{p \leq X}\sum_{k > \frac{\log{X}}{\log{p}}}\frac{p^{-k(\s+it)}}{k^{m+1}(\log{p})^{m}}\\
&= \sum_{p \leq X}\sum_{\frac{\log{X}}{\log{p}} < k \leq 3\frac{\log{X}}{\log{p}}}\frac{p^{-k(\s+it)}}{k^{m+1}(\log{p})^{m}}
+ O\l( \frac{X^{1 - 3\s}}{(\log{X})^{m}} \r).
\end{align*}
Here, we can write
\begin{align*}
&\bigg|  \sum_{p \leq X}\sum_{\frac{\log{X}}{\log{p}} < k \leq 3\frac{\log{X}}{\log{p}}}
\frac{p^{-k(\s+it)}}{k^{m+1}(\log{p})^{m}} \bigg|^2\\
&= \sum_{p \leq X}\sum_{\frac{\log{X}}{\log{p}} < k \leq 3\frac{\log{X}}{\log{p}}}\frac{p^{-2k\s}}{k^{2(m+1)}(\log{p})^{2m}}+\\
&\qquadt + \underset{(p_1, k_1) \not= (p_2, k_2)}
{\sum_{p_1 \leq X} \sum_{p_2 \leq X} \sum_{\frac{\log{X}}{\log{p_1}} < k_1 \leq 3\frac{\log{X}}{\log{p_1}}}
\sum_{\frac{\log{X}}{\log{p_2}} < k_2 \leq 3\frac{\log{X}}{\log{p_2}}}}
\frac{(p_1^{k_1}p_2^{k_2})^{-\s}(p_1^{k_1} / p_2^{k_2})^{-it}}{(k_1 k_2)^{m+1}(\log{p_1} \log{p_2})^{m}}.
\end{align*}
Therefore, it holds that
\begin{align*}
&\int_{0}^{T}\bigg|  \sum_{p \leq X}\sum_{\frac{\log{X}}{\log{p}} < k \leq 3\frac{\log{X}}{\log{p}}}
\frac{p^{-k(\s+it)}}{k^{m+1}(\log{p})^{m}} \bigg|^2 dt\\
&= T\sum_{p \leq X}\sum_{\frac{\log{X}}{\log{p}} < k \leq 3\frac{\log{X}}{\log{p}}}\frac{p^{-2k\s}}{k^{2(m+1)}(\log{p})^{2m}} +\\
&\qquadf \qquad +O\l( X^3 \l(\sum_{p \leq X}
\sum_{\frac{\log{X}}{\log{p}} < k \leq 3\frac{\log{X}}{\log{p}}}\frac{1}{p^{k\s} k^{m+1}(\log{p})^{m}} \r)^2 \r)\\
&\ll T\frac{X^{1-2\s}}{(\log{X})^{2m+1}}
+ \frac{X^{5 - 2\s}}{(\log{X})^{2(m+1)}}
\ll T\frac{X^{1-2\s}}{(\log{X})^{2m+1}}.
\end{align*}
Hence we have
\begin{align*}
&\int_{0}^{T}\bigg|\sum_{p \leq X}\frac{\Li_{m+1}(p^{-\s - it})}{(\log{p})^{m}}
- \sum_{2 \leq n \leq X}\frac{\Lam(n)}{n^{\s+it}(\log{n})^{m+1}}\bigg|^2 dt\\
&\ll \int_{0}^{T}\bigg|  \sum_{p \leq X}\sum_{\frac{\log{X}}{\log{p}} < k \leq 3\frac{\log{X}}{\log{p}}}
\frac{p^{-k(\s+it)}}{k^{m+1}(\log{p})^{m}} \bigg|^2 dt + T \frac{X^{2 - 6\s}}{(\log{X})^{2m}}
\ll T\frac{X^{1 - 2\s}}{(\log{X})^{2m+1}},
\end{align*}
which completes the proof of this lemma.
\end{proof}

\begin{proof}[Proof of Proposition \ref{ALD}]
By Lemma \ref{ML} and Lemma \ref{PLvM}, for $X \leq T^{1/135}$, we find that
\begin{align*}
&\frac{1}{T}\int_{14}^{T}\bigg| \tilde{\eta}_{m}(\s + it) 
- \sum_{p \leq X}\frac{\Li_{m+1}(p^{-\s - it})}{(\log{p})^{m}}\bigg|^2 dt\\
&\ll \frac{1}{T}\int_{14}^{T}\bigg| \tilde{\eta}_{m}(\s + it) 
- \sum_{2 \leq n \leq X}\frac{\Lam(n)}{n^{\s + it}(\log{n})^{m+1}} \bigg|^2 dt\\
&\qquadf + \frac{1}{T}\int_{14}^{T}\bigg| \sum_{p \leq X}\frac{\Li_{m+1}(p^{-\s - it})}{(\log{p})^{m}} 
- \sum_{2 \leq n \leq X}\frac{\Lam(n)}{n^{\s + it}(\log{n})^{m+1}} \bigg|^2 dt\\
&\ll \frac{X^{1 - 2\s}}{(\log{X})^{2m}}.
\end{align*}
By using this estimate, for any fixed $\e > 0$, we have
\begin{multline*}
\frac{1}{T}\meas\l\{ t \in [0, T] \; \Bigg{|} \; 
\bigg| \tilde{\eta}_{m}(\s + it) - \sum_{p \leq X}\frac{\Li_{m+1}(p^{-\s - it})}{(\log{p})^{m}} \bigg| \geq \e \r\}\\
\ll \frac{X^{1 - 2\s}}{\e^{2}(\log{X})^{2m}} + \frac{1}{T}.
\end{multline*}
Hence, for any $T \geq X^{135}$, it holds that
\begin{align*}
\frac{1}{T}\meas\l\{ t \in [0, T] \; \Bigg{|} \; 
\bigg| \tilde{\eta}_{m}(\s + it) - \sum_{p \leq X}\frac{\Li_{m+1}(p^{-\s - it})}{(\log{p})^{m}} \bigg| \geq \e \r\}
\rightarrow 0
\end{align*}
as $X \rightarrow +\infty$.
Thus, we obtain Proposition \ref{ALD}.
\end{proof}

\section{\textbf{Proof of Proposition \ref{DLm}}}
In this section, we prove Proposition \ref{DLm} by the method of \cite[VIII.3]{KV1992}, \cite{V1988}.
First of all, we will show the following elementary geometric lemma.
The following lemma is a special case of Lemma 3.2 in Takanobu's textbook \cite{Ta2013}.

\begin{lemma}\label{ELG}
Let $N$ be a positive integer larger than two.
Suppose that the positive numbers $r_1, r_2, \ldots, r_N$ satisfy the condition
\begin{align}	\label{AMP}
r_{n_0} \leq \sum_{\substack{ n =1 \\ n \neq n_0 } }^N r_n, 
\end{align}
where $r_{n _0} = \max\{ r_n \mid n = 1, 2, \ldots, N \}$. 
Then we have
\begin{align}	\label{SPC}
\left\{ \sum_{n =1}^{N}  r_n \exp ( - 2 \pi i \theta_n ) \in \CC ~\bigg{|}~ \theta_n \in [0,1) \right\}
=\left\{ z \in \mathbb{C} ~\bigg{|}~ | z | \leq \sum_{n =1}^N r_n \right\}.
\end{align}
\end{lemma}

We give a simpler proof than that in \cite{Ta2013} by using the following elementary geometric theorem 
on the existence of polygons.
Our proof is essentially the same as Takanobu's proof, 
but his proof seems complicated because he did not postulate the geometric theorem.

\begin{lemma}	\label{ECP}
Let $N$ be a positive integer larger than two. 
Suppose that the positive numbers $r_{1}, r_2, \dots, r_N$ satisfy condition \eqref{AMP}. 
Then, we can make an $N$-sided convex polygon with the lengths $r_1, r_2, \dots, r_N$.
\end{lemma}

The authors cannot find the reference in which the proof of this lemma is written.
Since this lemma seems well known and can be proved by an elementary argument, we omit the proof.

\begin{proof}[Proof of Lemma \ref{ELG}]
Let $\mathcal{L}$ and $\mathcal{R}$ be the set on the left and right hand side of \eqref{SPC} respectively.
The inclusion $\subset$ is trivial.
We take $z=r \exp(- 2 \pi i \phi) \in \mathcal{R}$ with $0<r \leq \sum_{n =1}^N r_i$ and $\phi \in \mathbb{R}$.
Put $r = r_0 = r_{N+1}$ for convenience.
We consider the line segments $R_1, R_2, \ldots, R_{N+1}$ with the lengths $r_1, r_2, \ldots, r_{N+1}$ respectively.
By assumption \eqref{AMP} and Lemma \ref{ECP},
we can make a convex polygon by connecting our line segments clockwise in order. 
For any $k = 1, 2, \ldots, N$, 
we define by $2 \pi \mu_k$ the positive angle between $R_{k-1}$ and $R_{k}$.
Then we have
\[
r = \sum_{n=1}^N r_n \exp \left( - 2 \pi i \left( \frac{n -1}{2} - \sum_{j =1}^n \mu_j \right)  \right).
\]
Hence we have
\[
z = \sum_{n =1}^N  r_n\exp\left(  - 2 \pi i \left( \frac{n -1}{2} - \sum_{j =1}^n \mu_j  + \phi \right) \right).
\]
Taking $\theta_n = \{ (n -1)/2 - \sum_{j=1}^n \mu_j + \phi \} \in [0,1)$ for any $n =1, 2, \ldots, n$, 
we have $z = \sum_{i =1}^N r_i \exp( - 2 \pi i \theta_i ) \in \mathcal{L}$.
Here $\{x\}$ means the fractional part of $x$.
We find that $0 \in \mathcal{L}$ by the similar argument.
This completes the proof.
\end{proof}

Next, we introduce the following definitions.
\begin{definition}
Let $m$ be a positive integer and $\mathcal{M}$ a finite subset of the set of prime numbers and $\underline{\theta} = (\theta_p)_{p \in \mathcal{M}} \in [0,1)^{\mathcal{M}}$. 
We define the functions $\phi_{m, \s, \mathcal{M}}$, $\Phi_{m, \s, \mathcal{M}}$ by
\begin{gather*}
\phi_{m, \mathcal{M}}(\s, \underline{\theta})
:=\sum_{p \in \mathcal{M}}\frac{\exp(- 2 \pi i \theta_p)}{p^{\sigma}(\log p)^m}, \\
\te_{m, \mathcal{M}}(\s, \underline{\theta})
:=\sum_{p \in \mathcal{M}} \frac{\Li_{m+1}(p^{ - \sigma}\exp(- 2 \pi i \theta_p))}{(\log{p})^{m}}
= \sum_{p \in \mathcal{M}}\sum_{k =1}^\infty\frac{\exp( -2 \pi i k \theta_p)}{k^{m+1}p^{k \sigma}(\log p)^m},
\end{gather*}
respectively.
\end{definition}

\begin{definition}
Let $p_n$ be the $n$-th prime number. 
Put 
\begin{align*}
\underline{\theta}^{(0)}
= \left(\theta_{p_n}^{(0)}\right)_{n \in \mathbb{N}}
=(0, 1/2 , 0 , 1/2, \ldots) \in [0, 1)^{\NN},
\end{align*}
and
\begin{align*}
\gamma_{m, \sigma} = \sum_{p}\sum_{k =1}^{\infty}\frac{\exp( -2 \pi i k \theta_p^{(0)} )}{k^{m+1}p^{k \sigma}(\log p)^m}.
\end{align*}
\end{definition}
Note that the series of the definition of $\gamma_{m, \sigma}$ is convergent.

\begin{proof}[Proof of Proposition \ref{DLm}]
Fix a complex number $a$ and $1/2 \leq \sigma <1$.
Let $U$ be a positive real parameter. 
We take a sufficiently large number $N = N(U, m, \s, a)$ for which the estimates  
\begin{gather*}
| a - \gamma_{m, \sigma} | \leq \sum_{p \in \mathcal{M}}\frac{1}{p^{\sigma}(\log p)^m},\\
\frac{1}{p_{\min}^{\sigma}( \log p_{\min})^m} 
\leq \sum_{p \in \mathcal{M}\setminus \{p_{\min}\}}\frac{1}{p^{\sigma}(\log p)^m}
\end{gather*}
are satisfied, where $\mathcal{M} = \mathcal{M}(U, N)$ is defined as $\Set{p}{p: \text{prime, }U < p \leq N}$, 
and $p_{\min}$ is the minimum of $\mathcal{M}$.
Note that the existence of such $N$ is guaranteed by $\sum_{p}\frac{1}{p^{\sigma}(\log p)^m}=\infty$.
Then, by Lemma \ref{ELG}, the function
\begin{align*}
\varphi_{m, \mathcal{M}}(\s, \cdot)
~:~
[0, 1)^{\mathcal{M}} \ni 
\underline{\theta}
\longmapsto
\varphi_{m, \mathcal{M}}(\s, \underline{\theta}) \in \left\{ z \in \mathbb{C} \; \Big{|} \; |z| \leq \sum_{p \in \mathcal{M}}\frac{1}{p^{\sigma}(\log p)^m} \right\}
\end{align*}
is surjective. 
Hence, there exists some $\underline{\theta}^{(1)} = \underline{\theta}(m, \sigma, U, N)^{(1)} = 
(\theta_p^{(1)})_{p \in \mathcal{M}} \in [0, 1)^\mathcal{M}$ such that
\begin{align*}
\phi_{m, \mathcal{M}}(\s, \underline{\theta}^{(1)})
= a - \gamma_{m, \sigma}.
\end{align*}
By using the prime number theorem, we find that
\begin{align*}
\te_{m, \mathcal{M}}(\s, \underline{\theta}^{(1)})
&= \phi_{m, \mathcal{M}}(\s, \underline{\theta}^{(1)}) 
+ \sum_{p \in \mathcal{M}}\sum_{k =2}^{\infty}\frac{\exp( -2 \pi i k \theta_p^{(1)} )}{k^{m+1}p^{k \sigma}(\log p)^m}\\
&= a - \gamma_{m ,\sigma} + O\left( \frac{1}{(\log U)^m} \right). 
\end{align*}
For any prime number $p$, we put
\begin{align*}
\theta_p^{(2)}
=
\begin{cases}
    \theta_p^{(0)} \quad \text{if} \quad p \notin \mathcal{M}, \\
    \theta_p^{(1)} \quad \text{if} \quad p \in \mathcal{M}.
  \end{cases}
\end{align*}
Then it holds that
\begin{align*}
&\sum_{p \leq N} \frac{\Li_{m+1}(p^{ - \sigma}\exp(- 2 \pi i \theta_p^{(2)}))}{(\log{p})^{m}}\\
=& \sum_{p \in \mathcal{M}}\frac{\Li_{m+1}(p^{ - \sigma}\exp(- 2 \pi i \theta_p^{(1)}))}{(\log{p})^{m}} 
+ \sum_{p \leq U}\frac{\Li_{m+1}(p^{- \sigma}\exp(- 2 \pi i \theta_p^{(0)}))}{(\log{p})^{m}}\\
=& \te_{m, \mathcal{M}}(\s, \underline{\theta}^{(1)})
+ \gamma_{m, \sigma} 
+ \sum_{p > U}\frac{\Li_{m+1}(p^{- \sigma}\exp(- 2 \pi i \theta_p^{(0)}))}{(\log{p})^{m}},
\end{align*}
and additionally, by using the prime number theorem and simple calculations of alternating series, 
\begin{align*}
\sum_{p > U}\frac{\Li_{m+1}(p^{- \sigma}\exp(- 2 \pi i \theta_p^{(0)}))}{(\log{p})^{m}}
&= \sum_{p > U}\frac{\exp(- 2 \pi i \theta_p^{(0)}))}{p^{\sigma}(\log{p})^{m}}
+ O\l( \sum_{p > U}\frac{1}{p^{2\s} (\log{p})^{m}} \r)\\
&\ll \frac{1}{(\log{U})^{m}}.
\end{align*}
Hence, by taking a sufficiently large $U = U(\e)$
and noting the continuity of the function 
$\sum_{p \leq N} \frac{\Li_{m+1}(p^\sigma\exp(- 2 \pi i \theta_p))}{(\log{p})^{m}}$ with respect to 
$(\theta_{p})_{p \leq N} \in [0, 1)^{\pi(N)}$, we obtain this proposition.
\end{proof}

\section{\textbf{Proof of Theorem \ref{DTtetam}}}

In this section, we prove Theorem \ref{DTtetam}. 
Here, we use the following lemma related with Kronecker's approximation theorem.

\begin{lemma}	\label{AKT}
Let $A$ be a Jordan measurable subregion of $[0, 1)^{N}$, and $a_1, \dots, a_N$ be real numbers linearly independent over $\mathbb{Q}$. 
Set, for any $T > 0$, 
\begin{align*}
I(T, A) = \l\{ [0, T] \mid (\{ a_1 t\}, \dots, \{a_{N} t \}) \in A \r\}.
\end{align*}
Then we have
\begin{align*}
\lim_{T \rightarrow +\infty}\frac{\meas(I(T, A))}{T} = \meas(A).
\end{align*}
\end{lemma}

\begin{proof}
This lemma is Theorem 1 of Appendix 8 in \cite{KV1992}
\end{proof}

Let us start the proof of Theorem \ref{DTtetam}.

\begin{proof}[Proof of Theorem \ref{DTtetam}]
Let $\e > 0$ be any small number, $a$ any fixed complex number, $\frac{1}{2} \leq \s < 1$, 
and let $T_0$ be any positive number.
Then, by Proposition \ref{DLm}, we can take a sufficiently large $M_0 = M_0( m, \sigma, a, \e)$ 
so that for any $M \geq M_0$, 
there exists some Jordan measurable subset $\Theta_{1}^{(M)} = \Theta_{1}^{(M)}(m, \s, a, \e, M)$ of $[0, 1)^{M}$ 
such that $\delta_{M} := \meas(\Theta_{1}^{(M)})>0$ and 
\begin{align*}
|S_{M}(\theta_{1}, \dots, \theta_{M};\s, m) - a| < \e
\end{align*}
for any $(\theta_{1}, \dots, \theta_{M}) \in \Theta_1^{(M)}$. 

Here, we define $S_{M}(\theta_1, \dots, \theta_{M};\s, m)$ and $S_{M, N}(\theta_{M+1}, \dots, \theta_{N};\s, m)$ by 
\begin{gather*}
S_{M}(\theta_{1}, \dots, \theta_{M}; \s, m)
= \sum_{n \leq M}\frac{\Li_{m+1}(p_{n}^{-\s} e^{-2\pi i \theta_{n}})}{(\log{p_{n}})^{m}}, \\
S_{M, N}(\theta_{M+1}, \dots, \theta_{N};\s, m)
= \sum_{M < n \leq N}\frac{\Li_{m+1}(p_{n}^{-\s}e^{-2\pi i\theta_{n}})}{(\log{p_{n}})^{m}}.
\end{gather*}
Then, we find that
\begin{align*}
&\int_{0}^{1} \cdots \int_{0}^{1}|S_{M, N}(\theta_{M+1}, \dots, \theta_{N}; \s, m)|^2d\theta_{M+1} \cdots d\theta_{N}\\
&= \int_{0}^{1} \cdots \int_{0}^{1}\l| 
\sum_{M < n \leq N}\sum_{k = 1}^{\infty}\frac{p_{n}^{-\s k}e^{-2\pi i k \theta_{n}}}{k^{m+1}(\log{p_{n}})^{m}}\r|^2 d\theta_{M+1} \cdots d\theta_{N}\\
&= \sum_{M < n_1 \leq N}\sum_{M < n_2 \leq N}\sum_{k_1 = 1}^{\infty}\sum_{k_2 = 1}^{\infty}\Bigg\{
\frac{(p_{n_1} p_{n_2})^{-\s k}}{(k_1 k_2)^{m+1}(\log{p_{n_1}} \log{p_{n_2}})^{m}} \times \\
&\qquadt \qquadf \times \int_{0}^{1} \cdots \int_{0}^{1} e^{-2\pi i (k_1\theta_{n_1} - k_2\theta_{n_2})}d\theta_{M+1} \cdots d\theta_{N} \Bigg\}\\
&= \sum_{M < n \leq N}\sum_{k = 1}^{\infty}\frac{1}{k^{2(m+1)}p_{n}^{2\s k}(\log{p_{n}})^{2m}}
\ll \sum_{M < n \leq N}\frac{1}{p_{n}(\log{p_n})^{2m}}.
\end{align*}
Note that the last sum tends to zero as $M \rightarrow + \infty$. 
Therefore, there exists some large number $M_{1} = M_{1}(m, \e)$ such that, for any $N > M \geq M_{1}$, it holds that
\begin{align*}
\meas\l(\l\{ (\theta_{M+1}, \dots, \theta_{N}) \in [0, 1)^{N-M} \mid |S_{M, N}(\theta_{M+1}, \dots, \theta_{N}; \s, m)| < \e \r\}\r) 
> \frac{1}{2}.
\end{align*}
Here we denote the set of the content of $\meas(\cdot)$ in the above inequality by $\Theta_2^{(M, N)} = \Theta_2^{(M, N)}(M, N, \e)$.

We put $M_2 = \max\{M_0, M_1\}$ and $\Theta_3 = \Theta_{1}^{(M_2)} \times \Theta_2^{(M_2, N)}$ for any $N > M_2$.
Then $\Theta_3$ is a subset of $[0, 1)^{N}$ satisfying $\meas(\Theta_{3}) > \delta_{M_2} / 2$. 
Hence, putting
\begin{align*}
\I(T)
= \l\{ t \in [T_{0}, T] \; \Bigg{|} \; \l( \l\{ \frac{t}{2\pi}\log{p_1} \r\}, \dots, \l\{ \frac{t}{2\pi}\log{p_N} \r\} \r) \in \Theta_3 \r\}
\end{align*}
and applying Lemma \ref{AKT}, for any positive integer $N > M_2$, there exists some large number $T_N  > T_{0}$ 
such that $\meas(\I(T)) > \delta_{M_2} T / 2$ holds for any $T \geq T_N$.
On the other hand, by Proposition \ref{ALD}, there exists a large number $N_{0} = N_{0}(\e, \delta_{M_2})$ such that
\begin{align*}
\meas\l\{ t \in [T_{0}, T] \; \Bigg{|} \; \bigg| \te_{m}(\s + it) 
- \sum_{n \leq N}\frac{\Li_{m+1}(p_{n}^{-\s - it})}{(\log{p_{n}})^{m}} \bigg| < \e \r\}
> (1 - \delta_{M_2} / 4)T
\end{align*}
for any $N \geq N_0$, $T \geq p_N^{135}$.

Therefore, for any $N \geq \max\{N_{0}, M_2 +1\}$, $T \geq \max\{T_N, p_N^{135}\}$, there exists some $t_{0} \in [T_{0}, T]$ such that
\begin{align*}
\l( \l\{ \frac{t_{0}}{2\pi}\log{p_1} \r\}, \dots, \l\{ \frac{t_{0}}{2\pi}\log{p_N} \r\} \r) \in \Theta_3,
\end{align*}
and
\begin{align*}
\l| \te_{m}(\s + it_{0}) - \sum_{n \leq N}\frac{\Li_{m+1}(p_{n}^{-\s - it_{0}})}{(\log{p_{n}})^{m}} \r| < \e.
\end{align*}
Then we have
\begin{align*}
&|\te_{m}(\s + it_{0}) - a|\\
&\leq \bigg| \te_{m}(\s + it_{0}) - \sum_{n \leq N}\frac{\Li_{m+1}(p_{n}^{-\s} e^{-it_{0}\log{p_n}})}{(\log{p_{n}})^{m}} \bigg|
+ \l| \sum_{n \leq M_{2}}\frac{\Li_{m+1}(p_{n}^{-\s} e^{-it_{0}\log{p_n}})}{(\log{p_{n}})^{m}} - a \r|\\
&\qquad + \l| \sum_{M_{2} < n \leq N} \frac{\Li_{m + 1}(p_{n}^{-\s} e^{-it_{0}\log{p_n}})}{(\log{p_{n}})^{m}} \r|
< 3\e.
\end{align*}
This completes the proof of Theorem \ref{DTtetam}. 
\end{proof}

\section{\textbf{Proof of Theorem \ref{DTetam}}}

In this section, we prove Theorem \ref{DTetam}. 
Here, we prepare the following lemma.

\begin{lemma}	\label{BFeY}
Let $\s \geq 1/2$ and $m$ be a positive integer.
Then we have
\begin{align*}
\eta_{m}(s) = Y_{m}(s) + O_{m}(\log{t}).
\end{align*}
\end{lemma}

\begin{proof}
This lemma is equation (2.2) in \cite{II2019}.
\end{proof}

\begin{proof}[Proof of Theorem \ref{DTetam}]
First, we show Theorem \ref{DTetam} in the case $m=1$.
If the number of zeros $\rho = \b + i\gamma$ of $\zeta(s)$ with $\b > \s$ is finite, then there exists a sufficiently large $T_{0}$ such that
$Y_{1}(\s+it) \equiv b$ for $t \geq T_{0}$, where $b$ is a positive real number.
Therefore, by Lemma \ref{BFQm}, we have
\begin{align*}
\int_{0}^{t}\log{\zeta(\s+it')}dt' = i\te_{1}(\s+it) + b
\end{align*}
for any $t \geq T_{0}$.
By this formula, we obtain
\begin{align*}
\l\{ \int_{0}^{t}\log{\zeta(\s+it')}dt' \mid t \in [0, \infty) \r\}
&\supset \l\{ \int_{0}^{t}\log{\zeta(\s+it')}dt' \mid t \in [T_{0}, \infty) \r\}\\
&= \l\{ i\te_{1}(\s+it) + b \mid t \in [T_{0}, \infty) \r\}.
\end{align*}
If a set $A \subset \CC$ is dense in $\CC$, then, for any $c_1\in \CC \setminus \{
0 \}$ and $c_2 \in \CC$, 
the set $\{ c_1a + c_2 \mid a \in A \}$ is also dense in $\CC$.
By this fact and Theorem \ref{DTtetam}, 
the set $\{ i\te_{1}(\s+it) + b \mid t \in [T_{0}, \infty) \}$ is dense in $\CC$.
Thus, the set $\l\{ \int_{0}^{t}\log{\zeta(\s+it')}dt' \mid t \in [0, \infty) \r\}$ is dense in $\CC$ 
under this assumption.

Next, for $m \in \ZZ_{\geq 2}$, we show the equivalence of (I) and (II).
The implication (I) $\Rightarrow$ (II) is clear 
since the equation $\eta_{m}(\s+it) = i^{m}\te_{m}(\s+it)$ holds by assuming (I).

In the following, we show the inverse implication (II) $\Rightarrow$ (I). 
By Lemma \ref{BFeY}, if (I) is false, then the estimate $|\eta_{m}(\s+it)| \gg_{m} t^{m-1}$ holds. 
Therefore, for some $T_2 > 0$, we have
\begin{align*}
\overline{\l\{\eta_{m}(\s+it) \mid t \in [T_{2}, \infty) \r\}} \subset \CC \setminus \l\{ z \mid |z| \leq 1 \r\}.
\end{align*}
Here, $\bar{A}$ means the closure of the set $A$. 
Additionally, we see that 
\begin{align*}
\l\{ z \mid |z| \leq 1 \r\} \not\subset \overline{\l\{\eta_{m}(\s+it) \mid t \in [0, T_{2}) \r\}}
\end{align*}
since $\mu\l(\overline{\l\{\eta_{m}(\s+it) \mid t \in [0, T_{2} \r\}}\r) = 0$, where $\mu$ is the Lebesgue measure in $\CC$.
Hence, if (I) is false, then the set $\l\{\eta_{m}(\s+it) \mid t \in [0, \infty) \r\}$ is not dense in $\CC$.
Thus, we obtain the implication (II) $\Rightarrow$ (I).
\end{proof}

\section{\textbf{Note}}
In this paper, we discussed the denseness of the values $\te_{m}(\s + it)$, $t \in [0, \infty)$ 
without any probabilistic argument. 
The authors have already obtained the weak convergent limit of a certain sequence of probability measures 
and given the same results by combining the method of this paper and a probabilistic argument. 
The method would be suitable for generalization to other zeta and $L$-functions.
Some people may be interested in the existence of the probability density function of the weak 
convergent limit because
if we could obtain it, 
we may obtain more deep conclusions such as \cite{HM1999}, \cite{L2011}, \cite{LLR2019}.
However, it is difficult to find the probability density function of the limits in 
general cases, and it is actively studied even today such as \cite{IM2011Q}, \cite{IM2011M}, \cite{MU2019}, \cite{M2020}, 
and also good survey \cite{M2018}.
Here, we will mention Matsumoto's work.
In \cite{M1990}, he showed that there exist weak convergent limits in a general class of zeta-functions 
as the first step of the generalization of the theorem of Bohr-Jessen \cite{BJ1930}. 
On the other hand, he could not prove the existence of probability density functions in the general case.
His study suggests that the discussion not involving the probability density function is 
sometimes useful when we consider the generalization to other zeta and $L$-functions.
From these backgrounds, the authors adopted the method not involving
the probability density function as the first step of the study of the value-distribution 
of iterated integrals of the logarithm of the Riemann zeta-function.
We shall have further deep arguments including the probability density function 
in \cite{EIM2020}.

\begin{ackname}
The authors would like to deeply thank Professor Kohji Matsumoto for many useful comments and suggestions.
They would also like to thank Mr Masahiro Mine for useful discussion.
The second author is supported by Grant-in-Aid for JSPS Research Fellow (Grant Number: 19J11223).
\end{ackname}


\begin{thebibliography}{99}


\bibitem{B1916} H. Bohr, Zur Theorie der Riemann'schen Zeta-funktion im kritischen Streifen, \textit{Acta Math.} \textbf{40} 
(1916), 67--100.


\bibitem{BC1914} H. Bohr and R. Courant, Neue Anwendungen der Theorie der Diophantischen Approximationen 
auf die Riemannsche Zetafunktion, \textit{J. Reine Angew. Math.} \textbf{144} (1914), 249--274.


\bibitem{BJ1930} H. Bohr and B. Jessen, \"{U}ber die Werteverteilung der Riemannschen Zetafunktion, Erste Mitteilung, 
\textit{Acta Math.} \textbf{54} (1930), 1--35; Zweite Mitteilung, ibid. \textbf{58} (1932), 1--55.


\bibitem{BH1995} E. Bombieri and D. A. Hejhal, On the distribution of zeros of linear combinations of Euler products, 
\textit{Duke Math. J.} \textbf{80} (1995), no.3, 821--862.


\bibitem{EIM2020} K. Endo, S. Inoue, and M. Mine, On the value distribution of iterated integrals of the logarithm of 
the Riemann zeta-function II: Probability density functions, in preparation.


\bibitem{GS2014} R. Garunk\u{s}tis and J. Steuding, On the roots of the equation $\zeta(s) = a$, 
\textit{Abh. Math. Semin. Univ. Hambg.} \textbf{84} (2014), 1--15.


\bibitem{HM1999} T. Hattori and K. Matsumoto, A limit theorem for Bohr-Jessen's probability measures of the Riemann zeta-function, 
\textit{J. Reine Angew. Math.} \textbf{507} (1999), 219--232.


\bibitem{IM2011Q} Y. Ihara and K. Matsumoto, On certain mean values and the value-distribution of logarithms of 
Dirichlet $L$-functions, \textit{Quart. J. Math.} (Oxford) \textbf{62} (2011), 637--677.


\bibitem{IM2011M} Y. Ihara and K. Matsumoto, On $\log{L}$ and $L'/L$ for $L$-functions and the associated ``$M$-functions": 
Connections in optimal cases, \textit{Moscow Math. J.} \textbf{11} (2011), 73--111.


\bibitem{II2019} S. Inoue, On the logarithm of the Riemann zeta-function and its iterated integrals, 
preprint, \texttt{arXiv:1909.03643}.


\bibitem{JW1935}B. Jessen and A. Wintner, Distribution functions and The Riemann zeta function, 
\textit{Trans. Amer. Math. Soc.} \textbf{38} (1935), no. 1, 48--88.


\bibitem{KV1992} A. A. Karatsuba and S. M. Voronin, The Riemann zeta-function, Translated from the Russian by Neal Koblitz, 
De Gruyter Expositions in Mathematics \textbf{5}, 
Walter de Gruyter \& Co., Berlin, 1992.

\bibitem{KN2012} E. Kowalski and A. Nikeghbali, Mod-Gaussian convergence and the value distribution of $\zeta(\frac{1}{2}+it)$ 
and related quantities, \textit{J. London Math. Soc.} (2) \textbf{86} (2012), 291--319.


\bibitem{L2011} Y. Lamzouri, Distribution of large values of zeta and $L$-functions, \textit{Int. Math. Res. Not. IMRN} (2011), 
no.23, 5449--5503.


\bibitem{LLR2019} Y. Lamzouri, S. Lester, and M. Radziwi\l\l, Discrepancy bounds for the distribution 
of the Riemann zeta-function and applications, \textit{J. Anal. Math.}, \textbf{139} (2019), no.2 453--494.


\bibitem{M2018} K. Matsumoto, On the theory of $M$-function, RIMS K{\^o}ky{\^u}roku, 
\textit{Profinite monodromy, Galois representations, and complex functions} \textbf{2120}, M. Kaneko (ed.), 2019, pp.153--165,  
also available at \texttt{arXiv:1810.01552}. 


\bibitem{M1990} K. Matsumoto, Value-distribution of zeta-functions, Analytic number theory (Tokyo, 1988), 178--187, 
\textit{Lecture Notes in Math.}, 1434, Springer, Berlin, 1990.

\bibitem{MU2019} K. Matsumoto, Y. Umegaki, On the density function for the value-distribution of automorphic $L$-functions, 
\textit{J. Number Theory}, \textbf{198} (2019), 176--199.


\bibitem{M2020} M. Mine, On certain mean values of logarithmic derivatives of $L$-functions and the related density functions, 
\textit{Funct. Approx. Comment. Math.}, \textbf{61} (2019), no.2, 405--435.


\bibitem{RA2011} M. Radziwi\l\l, Large deviations in Selberg's limit theorem, \texttt{arXiv:1108.5092}.


\bibitem{SC1946} A. Selberg, Contributions to the theory of the Riemann zeta-function, 
Avhandl. Norske Vid.-Akad. Olso I. Mat.-Naturv. Kl., no.1;
Collected Papers, Vol. 1, New York: Springer Verlag. 1989, 214--280.


\bibitem{S1992} A. Selberg, Old and new conjectures and results about a class of Dirichlet series, 
\textit{Proceedings of the Amalfi Conference on Analytic Number Theory}, 367--385, 1992.


\bibitem{Ta2013} S. Takanobu, \textit{Bohr-Jessen Limit Theorem}, Revisited, MSJ Memoirs 31, Math. Soc. Japan, 2013.

\bibitem{TT}  E. C. Titchmarsh, \textit{The Theory of the Riemann Zeta-Function}, 
Second Edition, Edited and with a preface by D. R. Heath--Brown, 
The Clarendon Press, Oxford University Press, New York, 1986.  


\bibitem{V1975} S. M. Voronin, Theorem on the `universality' of the Riemann zeta-function, 
\textit{Izv. Akad. Nauk SSSR, Ser. Mat.}, \textbf{39} (1975), 475--486 (in Russian). English transl.
: \textit{Math. USSR, Izv.}, \textbf{9} (1975), 443--445.


\bibitem{V1988}S. M. Voronin, $\Omega$-theorems in the theory of the Riemann zeta-function, (Russian), 
\textit{Izv. Akad. Nauk SSSR Ser.  Mat.}, \textbf{52} (1988), 
no. 2, 424--436 ; 
translation in Math. USSR--Izv. \textbf{32} (1989), no. 2, 429--442.
\end{thebibliography}
\end{document}